\documentclass[a4paper, 11pt, reqno]{amsart}
\usepackage[usenames,dvipsnames]{color}
\usepackage{amssymb, amsmath, latexsym, graphics, graphicx, pstricks}
\newcommand\GreenL{\mathcal{L}}
\newcommand\GreenR{\mathcal{R}}
\newcommand\GreenH{\mathcal{H}}
\newcommand\GreenD{\mathcal{D}}
\newcommand\GreenJ{\mathcal{J}}
\newcommand\trop{\mathbb{T}}

\newcommand\ft{\mathbb{FT}}

\newcommand\pft{\mathbb{PFT}}

\newtheorem{theorem}{Theorem}[section]
\newtheorem{lemma}[theorem]{Lemma}
\newtheorem{corollary}[theorem]{Corollary}

\newtheorem{proposition}[theorem]{Proposition}
\begin{document}
\title{Tropical matrix groups}
\maketitle
\begin{center}
    ZUR IZHAKIAN\footnote{
Fachbereich Mathematik,
Universit\"at Bremen,
D-28359 Bremen,
Germany and School of Mathematical Sciencies, Tel Aviv University, Tel Aviv 69978,
Israel. Email \texttt{zzur@math.biu.ac.uk} and \texttt{zzur@post.tau.ac.il}. Research supported by the
Alexander von Humboldt Foundation.} ,
    MARIANNE JOHNSON\footnote{School of Mathematics, University of Manchester,
Manchester M13 9PL, England. Email
\texttt{Marianne.Johnson@maths.manchester.ac.uk}. Research supported by
EPSRC Grant EP/H000801/1.}
    and
    MARK KAMBITES\footnote{School of Mathematics, University of Manchester,
Manchester M13 9PL, England. Email \texttt{Mark.Kambites@manchester.ac.uk}.
Research supported by EPSRC Grant EP/H000801/1.
Mark Kambites gratefully acknowledges the hospitality of 
\textit{Universit\"at Bremen} during a visit to Bremen.}

\keywords{}
\thanks{}

\end{center}

\numberwithin{equation}{section}

\begin{abstract}
We study the subgroup structure of the semigroup of finitary tropical
matrices under multiplication. We show
that every maximal subgroup is isomorphic to the full linear automorphism
group of a related tropical polytope, and that each of these groups is the
direct product of $\mathbb{R}$ with a finite group. We also show that there
is a natural and canonical embedding of each full rank maximal subgroup
into the group of units of the semigroup of matrices over the tropical
semiring with $-\infty$. Our results have numerous
corollaries, including the fact that every automorphism of a
projective (as a module) tropical polytope of full rank extends to
an automorphism of the containing space, and that every full rank
subgroup has a common eigenvector.
\end{abstract}

\section{Introduction}

Tropical algebra is the algebra of the real numbers (sometimes augmented 
with an extra element denoted by $-\infty$) under the operations of 
addition and maximum. It has applications in areas such as combinatorial 
optimisation and scheduling, control theory, and algebraic geometry to 
name but a few (see \cite{Butkovic10} for a survey of applications). Many 
problems arising from these application areas are naturally expressed 
using (tropical) linear equations, so much of tropical algebra concerns 
matrices.

In this paper, we study the full semigroup of real $n \times n$ square
matrices with tropical multiplication. An important step in understanding
tropical
algebra is to understand the maximal subgroups of this semigroup, in terms
of both their abstract group structure and the geometry of their natural
actions on tropical space. It is a basic fact of semigroup theory that every 
subgroup of a semigroup $S$ lies in a unique maximal subgroup. Moreover, 
the maximal subgroups of $S$ are precisely the \emph{$\GreenH$-classes} 
(see Section~\ref{sec_Green} below for definitions) of $S$ which contain 
idempotents element. Recent research of the authors into the structure
of idempotent matrices \cite{K_puredim} provides a useful basis for
studying subgroups.

In addition to this introduction, this article comprises seven sections. 
In Section 2 we introduce some preliminary definitions. In Section 3 we 
give a brief account of Green's relations for the semigroup of $n \times 
n$ tropical matrices and prove that every maximal subgroup $H$ is 
isomorphic to the automorphism group of a particular tropical polytope
(namely, the column space of any element of $H$). In Sections 4 and 5 we
summarise a number of results on idempotent tropical matrices.

A consequence of \cite{K_puredim} is that there is an extremely well-behaved
notion of \textit{rank} for tropical idempotents, and hence for maximal subgroups
of tropical matrices. The understanding of idempotents arising from
\cite{K_puredim} is markedly more comprehensive when the idempotents in
question have full rank. Accordingly, Section 6 reduces the problem of
understanding maximal subgroups to the full rank case, by showing that every
rank $k$ maximal subgroup of the semigroup of $n \times n$ matrices
is naturally isomorphic to a maximal subgroup of the semigroup of
$k \times k$ matrices.

Finally, Section~7 establishes our strongest results. We exhibit a
natural and canonical embedding of each full rank maximal subgroup
(and hence of each subgroup) into the group of units of the corresponding
semigroup of matrices over the tropical semiring with $-\infty$.
An analysis of this embedding allows us to show that each maximal subgroup
of rank $k$ is the direct product of $\mathbb{R}$ with a finite group of
permutation degree $k$ or less. Other corollaries of our results include
that every automorphism of a projective tropical polytope extends to an
automorphism of the containing space, and that every subgroup has a common
eigenvector.

The decomposition of maximal subgroups as direct products of $\mathbb{R}$
with finite groups establishes, in the case of matrices with real entries,
a conjecture of the second and third authors \cite{K_tropicalgreen}, which
states that every group of $n \times n$ tropical matrices has a torsion free
abelian subgroup of index $n!$ or less. This conjecture has been independently
proved in the
general case by Shitov \cite{Shitov12}. It is natural to ask exactly which
finite groups arise in these decompositions. In a companion paper of the 
second and third authors \cite{K_finitemetric}, we shall show that for
every finite group $G$, the group $G \times \mathbb{R}$ arises as a
maximal subgroup in sufficiently high dimensions.

Of course, it is also natural to ask about the subgroup structure of the
full semigroup of $n \times n$ matrices over the larger tropical semiring
including $-\infty$. Indeed, there are already a few interesting results
in this direction \cite{K_tropicalgreen,Shitov12}. In this article we have
chosen to focus on matrices with real entries, partly to avoid
technicalities and partly because some of the machinery and results we
employed have not yet been developed for the case with $-\infty$. However,
we believe the results and methods developed here should, with careful
application, suffice to permit a full understanding of the subgroup
structure in the more general case.

\section{Preliminaries}
\label{prelim}
We write $\ft$ for the set $\mathbb{R}$ equipped with the operations of maximum (denoted by $\oplus$) and addition (denoted by $\otimes$, by $+$ or simply by juxtaposition). Thus, we write $a \oplus b = \max(a,b)$ and $a \otimes b = ab = a + b$. It is readily verified that $\ft$ is an abelian group (with neutral element $0$) under $\otimes$ and a commutative semigroup of idempotents (without a neutral element) under $\oplus$, and that $\otimes$ distributes over $\oplus$. These properties mean $\ft$ has the structure of an \textit{idempotent semifield}.

It will sometimes be convenient to work with the extended tropical semifield $\trop =\ft\cup\{-\infty\}$, where we extend the definitions of $\oplus$ and $\otimes$ in the obvious way (namely, $a \oplus -\infty = -\infty \oplus a = a$ and $a \otimes -\infty = -\infty \otimes a = -\infty$, for all $a \in \trop$).

Let $M_n(\ft)$ denote the set of all $n \times n$ matrices with entries in 
$\ft$. The operations $\oplus$ and $\otimes$ can be extended in the 
obvious way to give corresponding operations on $M_n(\ft)$. In particular, 
it is easy to see that $M_n(\ft)$ is a semigroup with respect to tropical 
matrix multiplication. We shall see in the following sections that this 
semigroup has a rich and interesting structure.

We shall be interested in the space $\ft^n$ consisting of $n$-tuples $x$ with entries in $\ft$; we
write $x_i$ for the $i$th component of $x$. We call $\ft^n$ \textit{(affine) tropical $n$-space}.
The space $\ft^n$ admits
an addition and a scaling action of $\ft$ given by $(x\oplus y)_i = x_i \oplus y_i$ and
$(\lambda x)_i = \lambda (x_i)$ respectively. These operations give $\ft^n$
the structure of an \textit{$\ft$-module}\footnote{Some authors use the term \textit{semimodule}, to
emphasise the non-invertibility of addition, but since no other kind of module exists over $\ft$
we have preferred the more concise term.}. It also
has the structure of a lattice, under the partial order given by $x \leq y$ if $x_i \leq y_i$ for all $i$.

From affine tropical $n$-space we obtain \textit{projective tropical $(n-1)$-space}, denoted $\pft^{n-1}$,
by identifying two vectors if one is a tropical multiple of the other by an element of $\ft$. We
identify $\pft^{n-1}$ with $\mathbb{R}^{n-1}$ via the map
$$(x_1, \ldots, x_n) \mapsto (x_1-x_n, x_2 -x_n, \ldots,  x_{n-1} - x_n).$$

Submodules of $\ft^n$ (that is, subsets closed under tropical addition and scaling) are termed
\textit{(tropical) convex sets}. Finitely generated convex sets are called \textit{(tropical) polytopes}.
Since convex sets are closed under scaling, each convex set $X \subseteq \ft^n$ induces a subset of
$\pft^{n-1}$, termed the \textit{projectivisation} of $X$ and denoted $\mathcal{P}X$.

For $A \in M_n(\ft)$ we let $R(A)$ denote the tropical polytope in $\ft^n$ generated by the rows of $A$ and let $C(A)$ denote the tropical polytope in $\ft^n$ generated by the columns of $A$. We call these tropical polytopes the \emph{row space} and \emph{column space} of $A$ respectively.

A point $x$ in a convex set $X$ is called \textit{extremal in $X$} if the set
$$X \smallsetminus \lbrace \lambda \otimes x: \lambda \in \ft \rbrace$$
is a submodule of $X$. Clearly some scaling of every such extremal point must lie in every generating set for
$X$. In fact, every tropical polytope is generated by its extremal points considered up to scaling \cite{Butkovic07,Wagneur91}.
\section{Green's relations, idempotents and regularity}
\label{sec_Green}
Green's relations are five equivalence relations ($\GreenL$, $\GreenR$, $\GreenH$, $\GreenD$ and $\GreenJ$) and three partial orders ($\leq_\GreenR$, $\leq_\GreenL$ and $\leq_\GreenJ$), which can be defined on any semigroup, and which describe the structure of its maximal subgroups and principal left, right and two-sided ideals. We briefly recap the definitions here; for further details (including proofs of the claimed properties) we refer the reader to an introductory text such as \cite{Howie95}.

Let $S$ be any semigroup. If $S$ is a monoid, we set $S^1 = S$, and otherwise we denote by $S^1$ the monoid obtained by adjoining a new identity element $1$ to $S$. We define a binary relation $\leq_\GreenR$ on $S$ by $a \leq_\GreenR b$ if $a S^1 \subseteq b S^1$, that is, if either $a = b$ or there exists $q$ with $a = bq$. We define another relation $\GreenR$ by $a \GreenR b$ if and only if $a S^1 = b S^1$. It is straight-forward to check that $\GreenR$ is an equivalence relation, and $\leq_\GreenR$ is a preorder (a reflexive, transitive binary relation) which induces a partial order on the $\GreenR$-equivalence classes.

The relations $\leq_\GreenL$ and $\GreenL$ are the left-right duals of 
$\leq_\GreenR$ and $\GreenR$ (that is, $a \leq_\GreenL b$ if $S^1 a 
\subseteq S^1 b$, and $a \GreenL b$ if $S^1 a = S^1 b$). The relations 
$\leq_\GreenJ$ and $\GreenJ$ are two-sided analogues ($a \leq_\GreenJ b$ 
if $S^1 a S^1 \subseteq S^1 b S^1$, and $a \GreenJ b$ if
$S^1 a S^1 = S^1 b S^1$). The relations $\GreenH$ and $\GreenD$ are
described in terms of 
the $\GreenL$ and $\GreenR$ relations. The $\GreenH$ relation is the 
intersection of $\GreenL$ and $\GreenR$ (that is, $a \GreenH b$ if $a 
\GreenL b$ and $a \GreenR b$), whilst the $\GreenD$ relation can be defined
by $a \GreenD b$ if and only if there 
exists an element $c \in S$ such that $a \GreenR c$ and $c \GreenL a$. It 
can be shown that both $\GreenH$ and $\GreenD$ are equivalence relations.

The study of Green's relations for the full tropical matrix semigroups was initiated (in the case of $M_2(\trop)$) by the second and third authors \cite{K_tropicalgreen}. In \cite{K_tropd}, Hollings and the third author gave a complete description of the $\GreenD$-relation for $M_n(\ft)$, using the \textit{duality} between the row and column space of a tropical matrix. In \cite{K_tropj} the second and third authors described the equivalence relation $\GreenJ$ and pre-order $\leq_\GreenJ$ in in $M_n(\ft)$ and $M_n(\trop)$. The main results of these papers, for the case $M_n(\ft)$, are summarised in the following theorem (see \cite[Proposition 3.1 and Theorem 5.1]{K_tropd} and \cite[Theorem 5.3 and Theorem 6.1]{K_tropj} for full details and proofs).

\begin{theorem}
\label{thm_greenchar}
Let $A, B \in M_n(\ft)$.
\begin{itemize}
\item[(i)] $A \leq_{\GreenL} B$ if and only if $R(A) \subseteq R(B)$;
\item[(ii)] $A \GreenL B$ if and only if $R(A) = R(B)$;
\item[(iii)] $A \leq_{\GreenR} B$ if and only if $C(A) \subseteq C(B)$;
\item[(iv)] $A \GreenR B$ if and only if $C(A) = C(B)$;
\item[(v)] $A \GreenH B$ if and only if $R(A) = R(B)$ and $C(A) = C(B)$;
\item[(vi)] $A \GreenD B$ if and only if  $C(A)$ and $C(B)$ are isomorphic as $\ft$-modules;
\item[(vii)] $A \GreenD B$ if and only if  $R(A)$ and $R(B)$ are isomorphic as $\ft$-modules;
\item[(viii)] $A \leq_{\GreenJ} B$ if and only if  there exists a convex set $X \subseteq \ft^n$ such that $R(A)$ embeds linearly into $X$ and $R(B)$ surjects linearly onto $X$;
\item[(ix)] $A \leq_{\GreenJ} B$ if and only if  there exists a convex set $X \subseteq \ft^n$ such that $C(A)$ embeds linearly into $X$ and $C(B)$ surjects linearly onto $X$.
\item[(x)] $A \GreenJ B$ if and only if $A \GreenD B$.
\end{itemize}
\end{theorem}

Parts (i)-(v) of the above theorem are straight-forward, whilst parts (vi)-(x) require considerably more work (the proofs make use of the row-column duality alluded to above as well as some elementary topological arguments).

The following result follows immediately from the definitions. 

\begin{lemma}
\label{lem_relations}
Let $A,B \in M_n(\ft)$.
\begin{itemize}
\item[(i)] If $A \leq_{\GreenL} B$ then any linear relation between the \emph{columns} of $B$ induces the same relation between the \emph{columns} of $A$.
\item[(ii)] If $A \leq_{\GreenR} B$ then any linear relation between the \emph{rows} of $B$ induces the same relation between the \emph{rows} of $A$.
\end{itemize}
\end{lemma}

\begin{proof}
We prove (i), the proof of (ii) being dual. Let $A \leq_{\GreenL} B$. If $A = B$, the result holds trivially. Assume then that $A= XB$ for some $X \in M_n(\ft)$. Define $f:C(B)\rightarrow C(A)$ to be left multiplication by $X$. Thus $f$ is a linear map sending the $i$th column of $B$ to the $i$th column of $A$ and it follows that any relation between the columns of $B$ induces the corresponding relation between the columns of $A$.
\end{proof}

In fact, there is a correspondence between inclusions of row spaces and certain surjections of column spaces; see \cite{K_tropd} for full details. 

\begin{theorem}
\label{thm_HK}\cite[Theorem 4.2, Corollary4.3]{K_tropd}.
\begin{itemize}
\item[(i)] $R(A)\subseteq R(B)$ if and only if there there is a surjective linear morphism from $C(B)$ to $C(A)$ taking the $i$th column of $B$ to the $i$th column of $A$ for all $i$.
\item[(ii)] $R(A) =R(B)$ if and only if there is a linear isomorphism from $C(A)$ to $C(B)$ taking the $i$th column of $B$ to the $i$th column of $A$ for all $i$.
 \item[(iii)] $C(A)\subseteq C(B)$ if and only if there there is a surjective linear morphism from $R(B)$ to $R(A)$ taking the $i$th row of $B$ to the $i$th row of $A$ for all $i$.
\item[(iv)] $C(A) =C(B)$ if and only if there is a linear isomorphism from $R(A)$ to $R(B)$ taking the $i$th row of $B$ to the $i$th row of $A$ for all $i$.
    \end{itemize}
\end{theorem}

We shall require a few more semigroup theoretic definitions. Let $S$ be any semigroup. We recall that $e \in S$ is an \emph{idempotent} if $e^2 =e$, whilst $a \in S$ is \emph{(von Neumann) regular} if there exists $x \in S$ with $axa=a$. It is easy to see that $a$ is regular if and only if $a$ is $\GreenR$-related to an idempotent (dually, if and only if $a$ is $\GreenL$-related to an idempotent). Moreover, it is well known that every subgroup of a semigroup $S$ lies in a unique maximal subgroup, and that the maximal subgroups are precisely the $\GreenH$-classes of $S$ containing
idempotents.

A complete description of the idempotent elements of $M_2(\ft)$ was given in \cite{K_tropicalgreen} and it follows from the results given there that the semigroup $M_2(\ft)$ is regular (that is, every $\GreenR$ class and every $\GreenL$ class contains an idempotent) and each maximal subgroup is isomorphic to either $\mathbb{R}$ or $\mathbb{R} \times S_2$. For $n \geq 3$ it is known that the semigroups $M_n(\ft)$ are no longer regular. In \cite{K_puredim} the present authors gave a geometric characterisation of the regular elements of $M_n(\ft)$. In the present paper we turn our attention to the maximal subgroups of $M_n(\ft)$, in other words, the $\GreenH$-classes containing an idempotent tropical matrix. Given an idempotent tropical matrix $E$, we first note that the $\GreenH$-class containing $E$ is isomorphic to the group of $\ft$-module automorphisms of $C(E)$.
\begin{theorem}
\label{thm_aut}
Let $E$ be an idempotent in $M_n(\ft)$, with corresponding $\GreenH$-class
denoted by $H_E$, and let ${\rm Aut}(C(E))$ denote the group of $\ft$-module automorphisms of $C(E)$. Define $\psi: {\rm Aut}(C(E)) \rightarrow H_E$ by
$$\psi: f \mapsto (f(E_1) \cdots f(E_n)),$$
where $E_i$ denotes the $i$th column of $E$. Then $\psi$ is an isomorphism of groups. [Dually, the $\GreenH$-class of $E$ is isomorphic to the group of $\ft$-module automorphisms of $R(E)$.]
\end{theorem}

\begin{proof}
Let $f: C(E) \rightarrow C(E)$ be an $\ft$-module automorphism and let $A = \psi(f)$. Since $f$ is surjective, it is clear that $C(A)=C(E)$ and hence $A \GreenR E$, by Theorem~\ref{thm_greenchar}~(iii). Now, since $f$ is a linear isomorphism of column spaces, taking the $i$th column of $E$ to the $i$th column of $A$, it follows from Theorem~\ref{thm_HK}~(ii) that $R(A)=R(E)$, giving $A \GreenL E$ and hence $A \GreenH E$. Thus $\psi$ is well-defined.

We claim that $\psi$ is a homomorphism of groups. Indeed, let $f, g$ be $\ft$-module automorphisms of $C(E)$ and let $\psi(f)=A$ and $\psi(g)=B$. Then it is straight-forward to check that $f(x) = A \otimes x$ and $g(x) = B \otimes x$, for all $x \in C(E)$. Moreover, since $B \GreenH E$, we have $B \otimes E = B$ giving
\begin{eqnarray*}
\psi(f \circ g) &=& (f\circ g(E_1) \cdots f\circ g(E_n))\\
&=&(A \otimes B\otimes E_1 \cdots A \otimes B\otimes E_n)\\
&=&A \otimes B\otimes E = A \otimes B = \psi(f )\otimes \psi(g).
\end{eqnarray*}
In fact, we shall show that $\psi$ is an isomorphism.

We show first that $\psi$ is injective. Indeed, suppose $\psi(f) = \psi(g)$.
Then $f(E_i) = g(E_i)$ for all $i$. But the columns $E_i$ generate $C(E)$,
so by linearity it must be that $f = g$.

It remains to show that this homomorphism is surjective.
Let $A \in H_E$ and define $f: C(E) \rightarrow C(E)$ by $f(x) = A \otimes x$.
Since $H_E$ is a group, there exists $A' \in H_E$ such that
$A\otimes A' = E = A' \otimes A$. Since $E$ acts as the identity on $C(E)$,
it follows that $f$ must be bijective; define $f': C(E) \rightarrow C(E)$ by $f'(x) = A' \otimes x$ and suppose for contradiction that $f(x) = f(y)$ for some $x \neq y$. Then
$$x = E \otimes x = A' \otimes A \otimes x = f'(f(x)) = f'(f(y)) = A' \otimes A \otimes y = E \otimes y = y,$$
contradicting $x \neq y$. Thus $f \in {\rm Aut}(C(E))$ and it is clear that $\psi(f) = A$, giving that $\psi$ is surjective.
\end{proof}

\section{Dimension, projectivity, idempotents and regularity}
There are several important notions of dimension for a tropical convex set $X \subseteq \ft^n$. The \emph{tropical dimension} is the topological dimension of $X$, viewed as a subset of $\mathbb{R}^n$ with the usual topology. Note that, in contrast to the classical (Euclidean) case, tropical convex sets may have regions of different
topological dimension. We say that $X$ has \emph{pure} dimension $k$ if every open (within X with the induced topology) subset of X has topological dimension $k$. The \emph{generator dimension} of $X$ is the minimal cardinality of a generating subset, under the linear operations of scaling and addition. If $X$ is a polytope, this is equal to the number of extremal points of $X$ considered up to scaling \cite{Butkovic07,Wagneur91}. The \emph{dual dimension}~\cite{K_puredim} is the minimal cardinality of a generating set under scaling and the induced operation of greatest lower bound within the convex set. (Notice that, in general, the greatest lower bound of two elements within a convex set $X$ need not be the same as their component-wise minimum, which may not be contained in $X$.)

In \cite{K_puredim}, the present authors gave a characterisation of \emph{projectivity} for tropical polytopes in terms of the geometric and order-theoretic structure on these sets. We briefly recall that a module $P$ is called \emph{projective} if every morphism from $P$ to another module $M$ factors through every surjective module morphism onto $M$. One of the main results of \cite{K_puredim} can be summarised as follows.

\begin{theorem}\cite[Theorems 1.1 and 4.5]{K_puredim}.
\label{thm_IJKmain}
Let $X\subseteq \ft^n$ be a tropical polytope. Then the following are equivalent:
\begin{itemize}
\item[(i)] $X$ is projective as an $\ft$-module;
\item[(ii)] $X$ is the column space of an idempotent matrix in $M_n(\ft)$;
\item[(iii)] $X$ has pure dimension equal to its generator dimension and dual dimension.
\end{itemize}
\end{theorem}

Since all three notions of dimension coincide for projective polytopes, we define the \emph{dimension} of a projective tropical polytope to be this common value. We shall refer to projective polytopes of dimension $k$ as \emph{projective $k$-polytopes}. The following result is a consequence of \cite[Theorem 4.2]{K_puredim}.

\begin{proposition}\cite[Theorem 4.2]{K_puredim}.
\label{prop_IJK}
Let $X\subseteq \ft^n$ be a projective $k$-polytope. Then $X$ is isomorphic to the column space of a $k \times k$ idempotent matrix over $\ft$.
\end{proposition}

We note that projective $n$-polytopes in $\ft^n$ turn out to have a particularly nice structure:
\begin{theorem}\cite[Proposition 5.5]{K_puredim}.
\label{thm_minplusproj}
Let $X\subseteq \ft^n$ be a projective $n$-polytope. Then $X$ is min-plus (as well as max-plus) convex.
\label{npolytrope}
\end{theorem}
It is easily verified that any tropical polytope that is min-plus (as well as max-plus) convex must be convex in the usual (Euclidean) sense.

Numerous definitions of rank have been introduced and studied for tropical 
matrices (see for example \cite{Akian06,Develin05} for more details), 
mostly corresponding to different notions of ``dimension'' of the row or 
column space. In light of Theorem~\ref{thm_IJKmain}, we shall focus on the 
following three definitions of rank. The \emph{tropical rank} of a matrix 
is the tropical dimension of its row space (or equivalently, by 
\cite[Theorem 23]{Develin04} for example, its column space). It also has
a characterisation in terms of the computation of the matrix permanent
\cite{Develin05}. The \emph{row rank} is 
the generator dimension of the row space, which by \cite[Proposition 
3.1]{K_puredim} is also the dual dimension of the column space. Dually, 
the \emph{column rank} is the generator dimension of the column space and 
also the dual dimension of the row space. Whilst these three notions of 
rank can in general differ, it follows from 
Theorem~\ref{thm_IJKmain} that they must coincide for any 
\emph{idempotent} matrix. Thus we shall refer without ambiguity to the 
\emph{rank} of an idempotent matrix.

We define a scalar product operation $\ft^n \times \ft^n \rightarrow \ft^n$ on affine tropical $n$-space by setting
$$\langle x | y \rangle = {\rm max}\{\lambda \in \ft : \lambda \otimes x \leq y\}.$$
This is a \textit{residual} operation in the sense of residuation theory \cite{Blyth72}, and has been frequently employed in max-plus algebra. We recall the following result from \cite{K_finitemetric}.

\begin{lemma}\cite[Lemma 5.3]{K_finitemetric}
\label{lem_bracket}
Let $E$ be an idempotent element of $M_n(\ft)$. Let $E_1, \ldots, E_n$ denote the columns of $E$. If $E_{i,i} = 0$ then $E_{j,i} = \langle E_j | E_i \rangle$ for all $j$.
\end{lemma}

\section{Eigenvalues, eigenvectors and idempotents}
Let $A \in M_n(\ft)$ and let $\Gamma_A$ denote the corresponding weighted directed graph. We define the \emph{maximum cycle mean of $A$} to be the maximum average\footnote{By \textit{average}, we mean the
classical arithmetic mean, which is the tropical geometric mean.} weight of a path from a node to itself in $\Gamma_A$. Let $\lambda$ be the maximum cycle mean of $A$. Then the \emph{critical graph of $A$} consists of all nodes and edges involved in any path from a node to itself with average weight $\lambda$. The nodes occurring in the critical graph are called \emph{critical nodes}. It can be shown (see \cite{Butkovic10}, for example) that if $\lambda \leq 0$ then the following series converges to a finite limit, denoted $A^+$, in $M_n(\ft)$:
$$A \oplus A^2 \oplus \cdots \oplus A^n \oplus \cdots.$$

The following result is well known to experts in tropical mathematics (see \cite{Butkovic10} for details).

\begin{theorem}
\label{eigenval}
Let $A \in M_n(\ft)$ with corresponding weighted directed graph $\Gamma_A$. Let $\lambda$ denote the maximum cycle mean of $A$ and set $A_{\lambda} = -\lambda \otimes A$.  Then
\begin{itemize}
\item[(i)] The maximum cycle mean $\lambda$ is the unique eigenvalue of $A$.
\item[(ii)] The columns of $(A_{\lambda})^+$ labelled by the critical nodes form a generating set for the eigenspace of $A$.
\item[(iii)] Let $i$ and $j$ be critical nodes. The columns of $(A_{\lambda})^+$ labelled by $i$ and $j$ are tropical multiples of each other if and only if they occur in the same strongly connected component of the critical graph.
\item[(iv)] The eigenspace has generator dimension equal to the number of strongly connected components in the critical graph.
\end{itemize}
\end{theorem}

From now on let $E$ be an idempotent in $M_n(\ft)$. Theorem~\ref{eigenval} applied to $E$ has a number of interesting consequences.
\begin{corollary}
\label{cor_eigen}
Let $E \in M_n(\ft)$ be an idempotent. Then
\begin{itemize}
\item[(i)] $E$ has unique eigenvalue $0$ (which is the maximum cycle mean of $E$) and corresponding eigenspace $C(E)$.
\item[(ii)] $E_{i,i}=0$ if and only if $i$ is a node in the critical graph of $E$.
\item[(iii)] The columns of $E$ with diagonal entry $0$ form a generating set for the column space of $E$.
\item[(iv)] Every extremal point of $C(E)$ occurs up to scaling as a column of $E$ with $0$ in the diagonal position.
\item[(v)] The rows of $E$ with diagonal entry $0$ form a generating set for the row space of $E$.
\item[(vi)] Every extremal point of $R(E)$ occurs up to scaling as a row of $E$ with $0$ in the diagonal position.
\item[(vii)] The rank of $E$ is equal to the number of strongly connected components of the critical graph of $E$.
\item[(viii)] Each strongly connected component of the critical graph of $E$ is a complete subgraph of $\Gamma_E$.
\item[(ix)] If $i$ and $j$ are in the same strongly connected component of the critical graph then the $i$th column of $E$ is a multiple of the $j$th column of $E$.
\item[(x)] If $i$ and $j$ are in the same strongly connected component of the critical graph then the $i$th row of $E$ is a multiple of the $j$th row of $E$.
\end{itemize}
\end{corollary}

\begin{proof}
(i) Since $E$ is idempotent it is immediate that $E \otimes c = c$ for all columns $c$ of $E$, giving that $0$ is an eigenvalue of $E$, with corresponding eigenspace $C(E)$. By Theorem~\ref{eigenval}(i), $E$ has unique eigenvalue equal to the maximum cycle mean of $E$.

(ii) We note that the critical graph of $E$ consists of all nodes and edges involved in any zero-weighted path from a node to itself. If $E_{i,i}=0$ then, by definition, $i$ is a critical node of $E$. On the other hand, if $i$ is a critical node of $E$ then there is a path from $i$ to $i$ with weight zero. If this path has length $k$, then $E^{\otimes k}_{i,i} = 0$. Since $E$ is an idempotent, this gives $E_{i,i} =0$.

(iii) By Theorem~\ref{eigenval}(ii), the columns of $E$ labelled by the critical nodes form a generating set for the eigenspace of $E$. Since the eigenspace of $E$ is equal to the column space of $E$, the result now follows from, part (ii).

(iv) By the remarks in Section 2, some scaling of every extremal point must
lie in every generating set for $C(E)$.

(vii) Recall that the rank of an idempotent matrix may be defined to be the generator dimension of its column space. The result then follows from Theorem~\ref{eigenval}(iv).

(viii) Let $i$ and $j$ be in the same strongly connected component of the critical graph of $E$ and let $E_i$ and $E_j$ denote the $i$th and $j$th columns of $E$. By Theorem~\ref{eigenval}(iii) we see that $E_i = \alpha \otimes E_j$ (and hence $E_j = -\alpha \otimes E_i$). Moreover, since $i$ and $j$ are critical nodes we have $E_{i,i} = E_{j,j}=0$ by part (ii), giving
$$E_{i,j} + E_{j,i} = (\alpha \otimes E_{j,j}) + (-\alpha \otimes E_{i,i}) = \alpha + -\alpha = 0.$$
Thus we have a zero-weighted path from $i$ to itself via $j$. By definition, this cycle is in the critical graph of $E$. It then follows that each strongly connected component of the critical graph will be a complete graph.

(ix) This follows immediately from Theorem~\ref{eigenval}(iii).

Similar arguments hold for parts (v), (vi) and (x) by considering the row space of $E$ to be the column space of the idempotent $E^T$.
\end{proof}

Consider the set $C$ of critical nodes of $E$. There is an obvious equivalence relation on $C$, given by $i \sim j$ if and only if $i$ and $j$ are in the same strongly connected component of the critical graph. We call the equivalence classes of this relation the \emph{critical classes}. Notice that Corollary~\ref{cor_eigen} tells us that any set of representatives of the critical classes of $E$ yields a minimal generating set for the row [respectively, column] space of $E$. In fact, by Lemma~\ref{lem_relations}, any such set of representatives will give a generating set (not necessarily minimal) for the row [respectively, column] space of any matrix $\GreenR$-below [respectively, $\GreenL$-below] $E$.

\begin{corollary}
\label{cor_basis}
Let $A, E \in M_n(\ft)$, with $E$ idempotent and let $\{c_1, \ldots, c_k\}$ be a set of representatives of the critical classes of $E$.
\begin{itemize}
\item[(i)] If $A \leq_\GreenR E$ then the rows labelled by $c_1, \ldots, c_k$ form a generating set for the row space of $A$. If $A \GreenR E$ then this generating set is minimal.
\item[(i)] If $A \leq_\GreenL E$ then the columns labelled by $c_1, \ldots, c_k$ form a generating set for the column space of $A$. If $A \GreenL E$ then this generating set is minimal.
\end{itemize}
\end{corollary}

\begin{proof}
We prove part (i), the proof of part (ii) being dual. It follows immediately from Lemma~\ref{lem_relations} and Corollary~\ref{cor_eigen} that the rows labelled by $c_1, \ldots, c_k$ form a generating set for the row space of $A$. Now suppose that $A \GreenR E$. Then $C(A) = C(E)$ and Theorem~\ref{thm_IJKmain} gives that $C(A)$ has generator dimension equal to its dual dimension. By \cite[Proposition 3.1]{K_puredim}, the dual dimension of $C(A)$ is equal to the generator dimension of $R(A)$. Thus we see that the minimum cardinality of a generating set for $R(A)$ is equal to the minimal cardinality of a generating set for $C(A)=C(E)$, which by Corollary~\ref{cor_eigen} is equal to $k$.
\end{proof}

\section{A reduction to idempotents of full rank}
Let $E \in M_n(\ft)$ and suppose that $E$ has rank $k \leq n$. In this section we shall prove that the $\GreenH$-class of $E$, denoted by $H_E$, is
isomorphic, as a group, to the $\GreenH$-class of a $k \times k$ idempotent $F$ of full rank $k$. We begin by showing that each $\GreenD$-class contains an idempotent whose diagonal entries are all equal to $0$. Since the maximal subgroups in each $\GreenD$-class are all isomorphic, it follows that we may restrict attention to those idempotents with all diagonal entries equal to $0$. Given such an idempotent $E$, the main result of this section (Theorem~\ref{thm_fullrank}) constructs a $k \times k$ idempotent $F$ of full rank $k$ and a group isomorphism between the corresponding $\GreenH$-classes, $H_E$ and $H_F$. Throughout this section we shall make use of several results and proofs from \cite{K_puredim}.

\begin{theorem}
\label{thm_minplus}
Let $E\in M_n(\ft)$ be an idempotent. Then the column space of $E$ is min-plus convex if and only if there is an idempotent $F\in M_n(\ft)$ such that $F_{i,i}=0$ for all $i$ and $C(F)=C(E)$.
\end{theorem}

\begin{proof}
The statement is trivial for $n=1$. Thus we assume that $n\geq 2$.

Suppose first that $F$ is an idempotent with all diagonal entries equal
to zero. We show that $C(F)$ is min-plus convex, using the proof strategy
of \cite[Proposition 5.5]{K_puredim}.

Let $x, y \in C(F)$ and let $z$ be the component-wise minimum of $x$ and $y$. It suffices to show that $z \in C(F)$. Since $x, y \in C(F)$ and $F$ is idempotent it is immediate that $x=F \otimes x$ and $y = F \otimes y$. Moreover, since $z \leq x$ and $z \leq y$ we see that $F \otimes z \leq F \otimes x = x$ and $F \otimes z \leq F \otimes y = y$, giving $F \otimes z \leq z$. Using the fact that all the diagonal entries of $F$ are zero yields
$$(F \otimes z)_i = \bigoplus_{j=1}^n F_{i,j} \otimes z_j \geq F_{i,i} \otimes z_i = z_i,$$
so that $F \otimes z \geq z$. Thus we have shown that $F \otimes z = z$, giving $z \in C(F)$ as required.

Now suppose that $C(E)$ is min-plus (as well as max-plus) convex. To
construct an idempotent $F$ with the desired properties, we use a
strategy based on the proof of \cite[Theorem 1.4]{K_puredim}, although
some extra complications result from the fact that $E$ need not have full
rank. The key
idea is to construct a matrix whose $i$th column is the infimum
(in $\ft^n$) of all elements $u \in C(E)$ such that $u_i \geq 0$ and show
that it has the desired properties. Of course, we must first check that
such infima exist.

Let $i \in \{1, \ldots, n\}$ and for each coordinate $j \neq i$, consider the set
$$\{u_j: u \in C(E), u_i \geq 0\}.$$
It is easy to see that this set is non-empty and, since $C(E)$ is finitely generated, it has a lower bound and hence an infimum. It follows from the fact $C(E)$ is closed that this infimum will be attained. Choose an element $w_j \in C(E)$ such that $w_{j,j}$ attains this minimum and $w_{j,i}\geq 0$. By the minimality of $w_{j,j}$ and the fact that $C(E)$ is closed under scaling, we must have $w_{j,i} = 0$. Now let $F_i$ be the minimum of all the $w_j$'s. Then $(F_i)_i = 0$ and $F_i$ is clearly less than or equal to all vectors $u \in C(E)$ with $u_i \geq 0$. Thus we have shown that $F_i$ is a lower bound for all elements $u \in C(E)$ such that $u_i \geq 0$. Suppose that $z$ is another lower bound. By the min-plus convexity of $C(E)$, we see that each $F_i$ is itself an element of $C(E)$ with $(F_i)_i \geq 0$. Hence $z \leq F_i$. Thus $F_i$ is the infimum of all elements $u \in C(E)$ such that $u_i \geq 0$. Let $F$ be the matrix whose $i$th column is $F_i$. We shall prove that $F$ is an idempotent in $M_n(\ft)$ with $C(F)=C(E)$. Since the diagonal entries of $F$ are $F_{i,i} = (F_i)_i=0$, this will complete the proof.

We first show that $F$ must be idempotent. It follows from the definition of matrix multiplication that for all $i$ and $j$,
$$(F^2)_{i,j} \geq F_{i,j} \otimes F_{j,j} = F_{i,j} + 0 = F_{i,j}.$$
Now let $i, j, k \in \{1, \ldots, n\}$. It will suffice to show that
$F_{i,j} \geq F_{i,k} \otimes F_{k,j}$, since then
$F_{i,j} \geq \bigoplus_{k=1}^n F_{i,k} \otimes F_{k,j} = (F^2)_{i,j}$. Consider $w = (-F_{k,j})\otimes F_j = -(F_j)_k\otimes F_j$. Then $w \in C(E)$ and $w_k \geq 0$. Since $F_k$ is the infimum of all such points we have $F_k \leq w$. In particular, comparing the $i$th entries of these elements, we have $(-F_{k,j})\otimes F_{i,j} =(-F_{k,j})\otimes (F_j)_i= w_i \geq (F_k)_i = F_{i,k}$ and so $F_{i,j} \geq F_{i,k} \otimes F_{k,j}$, as required.

It remains to show that $C(F)=C(E)$. Since each column $F_i$ of $F$ is contained in $C(E)$ we have $C(F) \subseteq C(E)$. We shall prove that every extremal point of $C(E)$ occurs, up to scaling, as a column of $F$, so that $C(E)\subseteq C(F)$.

We first show that the columns of $F$ are extremal points of $C(E)$. Suppose
for a contradiction that $F_i$ is not an extremal point of $C(E)$. Then by definition we may write $F_i$ as a finite sum of elements in $C(E)$ which are not multiples of $F_i$, say $F_i = z_1 \oplus \cdots \oplus z_k$. Let $j$ be such that $0=F_{i,i}=z_{j,i}$. Since $z_j \in C(E)$, $z_{j,i} \geq 0$ and $F_i$ is the infimum of all such points, we must have $F_i \leq z_j$. On the other hand, $z_j$ forms part of a linear combination for $F_i$, giving $z_j \leq F_i$. So $F_i=z_j$, contradicting that $z_j$ is not a multiple of $F_i$.

Now let $x$ be an extremal point of $C(E)$ and let $E_1, \ldots, E_n$ denote the columns of $E$. By Corollary~\ref{cor_eigen} we know that $x$ is a multiple of some column $E_k$ with $E_{k,k}=0$. We shall show that $F_k = E_k$ and hence $x$ occurs up to scaling as a column of $F$, as required. Since $F_k$ is an extremal point of $C(E)$, there exists $j$ and $\lambda \in \ft$ such that $F_k = \lambda \otimes E_j$ and $E_{j,j}=0$. Moreover, since $F_{k,k}=0$ we find that
$$0=F_{k,k} = \lambda \otimes (E_j)_k = \lambda + E_{k,j},$$
and hence $\lambda = -E_{k,j}$. On the other hand, $E_{k,k} = 0$ and $E_k \in C(E)$, so by the definition of $F_k$ we have $F_k \leq E_k$. In other words, $ -E_{k,j} \otimes E_j \leq E_k$. Thus $-E_{k,j} \leq \langle E_j| E_k\rangle$. Now, since $E_{k,k}=E_{j,j}=0$ we may apply Lemma~\ref{lem_bracket} to find $\langle E_j| E_k\rangle = E_{j,k}$ and  $\langle E_k| E_j\rangle = E_{k,j}$. Thus $E_{k,j} + E_{j,k} \geq 0$. Since $E$ is idempotent, it follows that $-E_{k,j} = E_{j,k}$. Thus $\lambda = -\langle E_k| E_j\rangle =\langle E_j| E_k\rangle$ and hence $E_k = \lambda \otimes E_j =F_k$, as required.
\end{proof}

\begin{theorem}
Let $E \in M_n(\ft)$ be an idempotent. Then there is an idempotent
$E' \in M_n(\ft)$ such that $E \GreenD E'$ and $E'$ has all diagonal entries equal to $0$.
\end{theorem}

\begin{proof}
Let $E$ be an idempotent of rank $k$. By Proposition~\ref{prop_IJK},
$C(E) \cong C(F)$ for some $k \times k$ idempotent $F$. Since the various
notions of dimension described in Section 4 are isomorphism invariant, it is clear
that $F$ must have rank $k$ and hence, by Corollary~\ref{cor_eigen} for
example, $F$ has all diagonal entries equal to $0$. Moreover,
Theorem~\ref{thm_minplusproj} yields that $C(F)$ is min-plus convex.
Now let $X$ be the subset of $\ft^n$ consisting of all elements $v$ such
that: (i) the restriction of $v$ to the first $k$ entries yields an element
of $C(F)$ and (ii) all other entries of $v$ are equal to $v_k$. It is clear
from the definition that $X$ is a tropical polytope which is both max-plus
and min-plus convex. Moreover, it is easy to see that $X \cong C(F)$ and
hence $X$ is projective. By Theorem~\ref{thm_IJKmain}, $X$ is the column
space of some
$n \times n$ idempotent,  $X = C(E')$, say. Theorem~\ref{thm_minplus} and
the min-plus convexity of $X$ guarantee that $E'$ can be chosen with all
diagonal entries equal to zero. Hence we have shown that
$$C(E) \cong C(F) \cong X = C(E'),$$
where $E, E' \in M_n(\ft)$. By Theorem~\ref{thm_greenchar}(vi), this gives $E \GreenD E'$ as required.
\end{proof}
For the rest of this section we shall assume that $E$ is an idempotent matrix in $M_n(\ft)$ of rank $k$ whose diagonal entries are all equal to $0$. In the following proof we shall make use of the extended tropical semiring $\trop$, defined in Section 2.

\begin{theorem}
\label{thm_fullrank}
Let $E$ be an idempotent matrix in $M_n(\ft)$ whose diagonal entries are
all equal to $0$. Choose a fixed set of representatives of the critical
classes of $E$,  $\{c_1, \ldots, c_k\}$ say, and let $M$ be the $k \times n$
matrix, with entries in $\trop$ and rows indexed by $c_1, \dots c_k$,
defined by $M_{c_i, j} = 0 $ if $j = c_i$ and $M_{c_i, j} = -\infty$ otherwise. Then
 \begin{itemize}
\item[(i)] $F=M\otimes E \otimes M^T$ is an idempotent of rank $k$ in $M_k(\ft)$;
\item[(ii)] the map $\phi: A \mapsto M\otimes A \otimes M^T$ induces an isomorphism of groups between the $\GreenH$-class of $E$ and the $\GreenH$-class of $F$.
\end{itemize}
\end{theorem}

\begin{proof}
(i) We first note that, by definition, $F$ is the $k \times k$ submatrix of $E$, whose entries are labelled by $c_1, \ldots, c_k$. For simplicity, we index the entries of $F$ by $c_1, \ldots, c_k$ so that $F_{c_i,c_j} = E_{c_i,c_j}$. Since each $c_i$ is a critical node, it is immediate that all diagonal entries of $F$ are equal to $0$, giving
$$(F \otimes F)_{c_i, c_j} \geq F_{c_i,c_i} \otimes F_{c_i, c_j} = F_{c_i,c_j}.$$
On the other hand, $(F \otimes F)_{c_i, c_j}$ is equal to the maximum weight of a path of length 2 from node $c_j$ to node $c_i$ via one of $c_1, \ldots, c_k$. This is clearly bounded above by the maximum weight of a path of length 2 from node $c_j$ to node $c_i$ via \emph{any} node $1, \ldots, n$. Thus $(F \otimes F)_{c_i, c_j} \leq (E \otimes E)_{c_i,c_j} = E_{c_i,c_j} = F_{c_i,c_j}$, as required.

(ii) We note that $\phi$ maps each $n \times n$ matrix $A$ to the $k \times k$ submatrix of $A$, with entries labelled by $c_1, \ldots, c_k$ and, in particular, $F=\phi(E)$.  We write $H_E$ to denote the $\GreenH$-class of $E$ in $M_n(\ft)$ and $H_F$ to denote the $\GreenH$-class of $F$ in $M_k(\ft)$. Let $A \in H_E$. By Corollary \ref{cor_basis}, the columns of $A$ labelled by $c_1, \ldots, c_k$ form a minimal generating set for the column space of $A$ and the rows of $A$ labelled by $c_1, \ldots, c_k$ form a minimal generating set for the row space of $A$. This gives $C(AM^T) = C(A) = C(E)$ and $R(MA) = R(A) = R(E)$. Let $\nu: C(E) \rightarrow C(F)$ and $\rho: R(E) \rightarrow R(F)$ be the linear maps given by $\nu: v \mapsto M \otimes v$ and $\rho: v \mapsto v\otimes M$. Since $\nu$ maps the columns of $E$ labelled by $c_1, \ldots, c_k$ to the columns of $F$ labelled by $c_1, \ldots, c_k$, we see that $\nu$ is onto. Similarly, $\rho$ maps the rows of $E$ labelled by $c_1, \ldots, c_k$ to the rows of $F$ labelled by $c_1, \ldots, c_k$, giving that $\rho$ is onto. (In fact, it is straight-forward to check that these maps are isomorphisms of $\ft$-modules.) It follows that $C(F) = \nu(C(E)) = \nu(C(AM^T)) = C(MAM^T)= C(\phi(A))$ and $R(F) = \rho(R(E)) = \rho(R(MA)) = R(MAM^T)=R(\phi(A))$. Thus, for all $A \in H_E$, we have shown that $\phi(A) \GreenH F$. In other words, $\phi(A) \in H_F$ for all $A \in H_E$ so that $\phi$ restricts to a map from $H_E$ to $H_F$. Moreover, it follows from the fact that $\nu$ and $\rho$ are isomorphisms that $\phi$ is injective. We claim that $\phi: H_E \rightarrow H_F$ is an isomorphism of groups.

Let $A, B \in H_E$. We must show that $\phi(A \otimes B) = \phi(A) \otimes \phi(B)$. In other words, we want to show that $MABM^T = MAM^TMBM^T$ for all $A, B \in H_E$. It is straight-forward to verify from the definition of matrix multiplication that this amounts to proving the following claim:

{\bf Claim: }For every pair $i, j \in \{1, \ldots, k\}$ there exists $t \in \{1, \ldots, k\}$ such that $(A\otimes B)_{c_i,c_j} = A_{c_i, c_t} \otimes B_{c_t,c_j}.$

Suppose for contradiction that that this is not the case. Then there exists $s \notin \{c_1, \ldots, c_k\}$ such that
$$(A\otimes B)_{c_i,c_j} = A_{c_i, s} \otimes B_{s,c_j} >  A_{c_i, c_l} \otimes B_{c_l,c_j},$$
for all $l \in \{1, \ldots, k\}$. Since $E$ has all diagonal entries equal to $0$, we know that $s$ occurs in the same strongly connected component as $c_t$ for some $t \in \{1, \ldots, k\}$. Corollary~\ref{cor_eigen} yields that column $s$ of $E$ must be $\alpha$ times column $c_t$ of $E$, whilst row $s$ of $E$ must be $-\alpha$ times row $c_t$ of $E$, for some $\alpha \in \ft$. Since $A$ and $B$ are $\GreenH$-related to $E$, it follows from Lemma~\ref{lem_relations} that we also have column $s$ of $A$ is equal to $\alpha$ times column $c_t$ of $A$ and row $s$ of $B$ is equal to $-\alpha$ times row $c_t$ of $B$. Thus
$$(A\otimes B)_{c_i,c_j} = A_{c_i, s} \otimes B_{s,c_j} =  \alpha \otimes A_{c_i, c_t} \otimes -\alpha \otimes B_{c_t,c_j} = A_{c_i, c_t} \otimes B_{c_t,c_j},$$
giving a contradiction. So $\phi$ is a homomorphism of groups.

It remains to show that $\phi$ is surjective. Let $N$ denote the $n \times k$
matrix defined by $N_{s, c_t} = 0 $ if $s = c_t$, $N_{s, c_t} = E_{s,c_t} $ if $s \notin \{c_1, \ldots, c_k\}$ and $N_{s, c_t}  = -\infty$ otherwise.
Let $P$ denote the $k \times n$ matrix defined by $P_{c_t, s} = 0 $ if $s = c_t$, $P_{c_t, s} = E_{c_t,s} $ if $s \notin \{c_1, \ldots, c_k\}$ and $P_{c_t, s}  = -\infty$ otherwise. Notice that
$$\phi(NKP) = M(NKP)M^T =(MN)K(PM^T)= K,$$
for all $K \in M_k(\ft)$. Let $G \in H_F$. We claim that $NGP \in H_E$, hence giving that $\phi$ is a surjection.

By the definition of $P$, the columns of $NGP$ labelled by $c_1, \ldots, c_k$ are the columns of $NG$, and all other columns are linear combinations of these. Thus $C(NGP) = C(NG)$. Since $G \in H_F$ we have that $FG=G$, giving $C(NG)=C(NFG)$ and it is clear that $C(NFG) \subseteq C(NF)$. Finally, it is straight-forward to check from the definitions of $N$ and $F$ that $NF$ is the $n \times k$ submatrix of $E$ whose columns are labelled by $c_1, \ldots, c_k$. This gives that $C(NF)=C(E)$ and hence we have shown that $C(NGP) \subseteq C(E)$.

Consider the linear map $\alpha: \ft^k \rightarrow \ft^n$ given by $v \mapsto N \otimes v$. It is clear from the definition of $N$ that $\alpha$ is injective. Moreover $\alpha$ maps the $c_i$th column of $G$ to the $c_i$th column of $NGP$, inducing a map $\alpha': C(G) \rightarrow C(NGP)$. As noted above, the columns of $NGP$ labelled by $c_1, \ldots, c_k$ form a generating set for $C(NGP)$ and hence $\alpha'$ is onto. This shows that $C(G) \cong C(NGP)$ via $\alpha'$. Hence we have $C(NGP) \cong C(G) =C(F) \cong C(E)$ and $C(NGP) \subseteq C(E)$. Since $C(NGP), C(E) \subseteq \ft^n$, the only way this can happen is if $C(NGP) = C(E)$. Thus $NGP \GreenR E$.

Dual arguments will show that $R(NGP) \subseteq R(E)$ and $R(E) \cong R(NGP)$, giving $NGP \GreenL E$ and hence $NGP \GreenH E$, as required.
\end{proof}

\section{The $\GreenH$-class of an idempotent of full rank}

Let $E$ be an idempotent in $M_n(\ft)$ and let $H_E$ denote the 
$\GreenH$-class of $E$. We can consider the matrices in $H_E$ as maps from 
$\ft^n$ to $\ft^n$, acting by left multiplication, and it follows from 
Theorem~\ref{thm_aut} that these maps restrict to $\ft$-module 
automorphisms of $C(E)$. We shall show that when $E$ has rank $n$, these 
automorphisms are affine linear maps, \textbf{in the classical sense}. It 
then follows that every automorphism of $C(E)$ extends to an automorphism 
of $\ft^n$.

We note that the \emph{boundary} of $C(E)$ is the set of all points
$y \in C(E)$ for which the equation $E \otimes x = y$ has multiple solutions.
Since $E$ acts trivially on $C(E)$, it follows that for every 
boundary point $y$ there exists some exterior point $z \notin C(E)$ such 
that $E \otimes z = y$. In fact, it is easy to see that every exterior 
point must be mapped to the boundary.

We shall need the following fact about the action of full rank idempotents on tropical $n$-space, which follows from results in
\cite{Butkovic10}.

\begin{lemma}
\label{lem_exterior}
Let $E$ be an idempotent of rank $n$ in $M_n(\ft)$, and consider
the column space $C(E)$ as a subset of $\mathbb{R}^n$ equipped with the usual
topology. Then left multiplication by $E$ maps all points exterior to $C(E)$
onto the boundary of $C(E)$.
\end{lemma}
\begin{proof}
Suppose false for a contradiction, and let $x \notin C(E)$ be such that
$Ex$ does not lie on the boundary of $C(E)$. Certainly $Ex \in C(E)$, so it
must be that $Ex$ lies in the interior of $C(E)$.
 Also since $E$ is idempotent, it has eigenspace
$C(E)$ with eigenvalue $0$, and $E = E^k$ for all $k$. Thus, by
\cite[Theorem~6.2.14]{Butkovic10}, each point in the interior of $C(E)$ has
a unique preimage under the action of $E$. But since $E$ is idempotent
we have $Ex = E(Ex)$, so $x = Ex$, which contradicts the fact that
$x \notin C(E)$.
\end{proof}

Now let $A \in H_E$. Thus $C(A)=C(E)$. Since $E$ has rank $n$, it follows that the columns of $A$ give a minimal generating set for $C(E)$. Since
this minimal generating set is unique up to permutation and scaling, we can find a permutation $\sigma$ and scalars $\lambda_1, \ldots, \lambda_n$ such that $A_i = \lambda \otimes E_{\sigma(i)}$ for all $i$. It now follows easily from \cite[Corollary 3.1.3]{Butkovic10}, for example, that the equation $A \otimes x = y$ has a unique solution if and only if the equation $E \otimes x = y$ has a unique solution. Thus it is clear that every boundary point $y$ of $C(E)$ is such that the equation $A \otimes x = y$ does \emph{not} have a unique solution. Since $A$ acts on $C(E)$ by an automorphism, it follows that for every boundary point $y$ there exists some exterior point $z \notin C(E)$ such that $A \otimes z = y$.

\begin{lemma}
\label{lem_Amaps}
Let $E$ be an idempotent of rank $n$ in $M_n(\ft)$, and let $A$ be an element in the $\GreenH$-class of $E$. Define $\phi_A: \ft^n \rightarrow \ft^n$ to be the map given by left multiplication by $A$. Then
\begin{itemize}
\item[(i)] $\phi_A$ maps interior points of $C(E)$ to interior points;
\item[(ii)] $\phi_A$ maps boundary points of $C(E)$ to boundary points and
\item[(iii)] $\phi_A$ maps all points exterior to $C(E)$ onto the boundary of $C(E)$.
\end{itemize}

[Dually, right multiplication by $A$ induces an $\ft$-module automorphism of $R(E)$, mapping interior points to interior points and boundary points to boundary points.]
\end{lemma}

\begin{proof}
It is clear that the image of $\phi_A$ is $C(E)$ and it follows from Theorem~\ref{thm_aut} that $\phi_A$ restricts to an automorphism of $C(E)$. Since $H_E$ is a group, there exists $A' \in H_E$ such that $AA'=A'A=E$ and hence $\phi_A$ and $\phi_{A'}$ are mutually inverse on $C(E)$.

(i) Let $x \in C(E)$ be an interior point of the column space. Suppose for contradiction that $\phi_A$ maps $x$ to some point $y$ on the boundary of $C(E)$. Then $\phi_{A'}$ must map $y$ back to $x$, so that $Ex = A'Ax = A'y = x$.
Consider the equation $Az=y$. It is easy to see that $x$ must be the unique
solution (since $\phi_A$ is an automorphism of $C(E)$, no other element of $C(E)$ can be a solution; if $z \notin C(E)$ were a solution, then $Ez=A'Az=A'y=x$, contradicting Lemma \ref{lem_exterior}). However, we have seen that given any point $y$ on the boundary there exists some $z \notin C(A)$ such that $Az=y$.

(ii) It follows immediately that $A$ must map boundary points to boundary points too; if $A$ maps a boundary point $y$ to an interior point $x$, then $A'$ must map $x$ back to $y$, contradicting part (i).

(iii) Finally, let $z \in \ft^n$ with $z \notin C(E)$ and suppose for contradiction that $A$ maps $z$ to an interior point $x$. By Lemma \ref{lem_exterior} we know that $E$ must map $z$ to a point on the boundary, $y$ say. Now $A'x = A'Az = Ez=y$, giving that $A'$ maps an interior point to a
boundary point, contradicting part (i).
\end{proof}

Recall that $\trop =\ft\cup\{-\infty\}$. We briefly consider the semigroup $M_n(\trop)$. It is clear that this is a monoid, whose identity element is the $n \times n$ matrix with $0$ entries on the diagonal and $-\infty$ entries off the diagonal. It is well known that the units of $M_n(\trop)$ are precisely the \emph{tropical monomial matrices}, that is, those matrices with exactly one entry not equal to $-\infty$ in each row and in each column. Thus it is clear that every unit has the form
$D(\lambda_1, \ldots, \lambda_n)P_{\sigma}$, where $D(\lambda_1, \ldots, \lambda_n)$ is a diagonal matrix with entries $\lambda_1, \ldots, \lambda_n$ and $P_{\sigma}$ is a tropical permutation matrix whose $i$th row has a $0$ in the $\sigma(i)$th position and $-\infty$ entries elsewhere. We shall now show that given an idempotent $E$ of rank $n$ in $M_n(\ft)$, the corresponding $\GreenH$-class $H_E$ is isomorphic to a certain subgroup of the group of units in $M_n(\trop)$.

\begin{theorem}\label{thm_upstairs}
Let $E$ be an idempotent of rank $n$ in $M_n(\ft)$ and let $H_E$ denote
the $\GreenH$-class of $E$. Let $G_E$ be the set of all units
$G \in M_n(\trop)$ which commute with $E$. Then there is a group isomorphism:
$$G_E \to H_E,  \ G \mapsto GE = EG.$$
\end{theorem}
\begin{proof}
Let $\gamma: G_E \rightarrow H_E$ be the map defined by $G\mapsto GE = EG$.
It is easy to see that $\gamma$ is a homomorphism of groups; indeed, for $G,H \in G_E$,
$$\gamma(GH)= E(GH)=E^2(GH)= E(EG)H = (EG)(EH)=\gamma(G)\gamma(H).$$

For injectivity, suppose $EG = EH$ for some units $G,H \in G_E$. Then
$EGH^{-1} = E$. Note that since $GH^{-1}$ is a monomial matrices, each
column of $EGH^{-1}$ is a scaling of a permutation of a column of $E$.
But since $E$ has rank $n$, no column of $E$ is a scaling of another
column. It follows that $GH^{-1}$ must be the identity element, giving
that $G=H$.

It remains to show that $\gamma$ is surjective. Let $A \in H_E$. By
Corollary~\ref{cor_basis}, the columns of $A$ provide a minimal
generating set for the column space of $E$. It follows that we may choose
a unit $G \in M_n(\trop)$ such that $A = EG$. Let $x$ be an interior
point of $C(E)$. By Lemma \ref{lem_Amaps}, $A$ maps $x$ to an interior
point. It follows that $G$ must also map $x$ to an interior point; if
not then, by Lemma~\ref{lem_exterior}, $EG$ maps $x$ to the boundary
of $C(E)$, contradicting $A=EG$. Thus it is easy to see that $Gx = Ax$
for all interior points $x$. Since $C(E)$ has pure dimension, every
boundary point is a limit of interior points. Since $G$ and $A$ are
continuous maps, it follows that $Gx = Ax$ for all $x \in C(E)$. In
particular, by Theorem~\ref{thm_aut}, the action of $G$ restricted to $C(E)$ is an
$\ft$-module automorphism. Now for all $x \in C(E)$ we have
$EGx=Gx=GEx$, giving
$$EG=A=AE=EGE=GEE=GE.$$
\end{proof}

Our main theorems combine to prove some results which may be of independent
interest.

\begin{theorem}\label{thm_extendaffine}
Every automorphism of a projective $n$-polytope in $\ft^n$
\begin{itemize}
\item[(i)] extends to an automorphism of $\ft^n$; and
\item[(ii)] is a (classical) affine linear map.
\end{itemize}
\end{theorem}
\begin{proof}
Let $X \subseteq \ft^n$ be a projective $n$-polytope and let
$\phi : X \to X$ be an $\ft$-module automorphism. By Theorem~\ref{thm_IJKmain},
$X = C(E)$ for some full rank idempotent $E \in M_n(\ft)$. By
Theorem~\ref{thm_aut}, there is a matrix $A \in H_E$ such that
$\phi(x) = Ax$ for all $x \in C(E) = X$. By Theorem~\ref{thm_upstairs}
there is a unit $G \in M_n(\trop)$ such that $A=EG=GE$. Now for
any $x \in X = C(E)$ we have
$$Gx = GEx = Ax = \phi(x),$$
so the map
$$\ft^n \to \ft^n, \ x \mapsto Gx$$
is an automorphism of $\ft^n$ extending $\phi$ as required to establish (i).

Now since $G$ is a monomial matrix, we have

$$G_{i,j}=\begin{cases}
\lambda_i &\mbox{ if }j=\sigma(i)\\
-\infty &\mbox{otherwise}.
\end{cases}$$
Thus, for all $x \in C(E)$,
$$(A \otimes x)_i = (GE\otimes x)_i = (G \otimes x)_i =\lambda_i \otimes x_{\sigma(i)} = x_{\sigma(i)} + \lambda_i,$$
giving $\phi_A(x) = P \cdot x + \lambda$, where $P$ is the (classical) permutation matrix corresponding to $\sigma$ and $\lambda = (\lambda_1, \ldots, \lambda_n)^T.$
\end{proof}

Since an $\ft$-module automorphism $f$ of $C(E)$ respects scaling, it induces
a well-defined map $\hat{f}$ on the projectivisation of $C(E)$. It turns out
that, in the case $E$ is an idempotent matrix of full rank, this map too is
affine linear.

\begin{corollary}
Let $E$ be an idempotent of rank $n$ in $M_n(\ft)$, let $A$ be an element
in the $\GreenH$-class of $E$ and let $\widehat{\phi_A}: \mathcal{P}C(E) \rightarrow \mathcal{P}C(E)$ denote the corresponding map induced by left multiplication by $A$. Then $\widehat{\phi_A}$ is a (classical) affine linear map on $\mathcal{P}C(E)$, regarded as a subset of $\mathbb{R}^{n-1}$.
\end{corollary}

\begin{proof}
By Theorem~\ref{thm_extendaffine}, $\phi_A(x) = P \cdot x + \lambda$, where $P$ is a (classical) permutation matrix and $\lambda = (\lambda_1, \ldots, \lambda_n)^T$ is a constant vector. Let
$$P_{i,j} = \begin{cases}
1 &\mbox{ if }j=\sigma(i)\\
0 &\mbox{otherwise},
\end{cases}$$
for some permutation $\sigma \in S_n$. Then
$$\phi_A(x) = \left(\begin{array} {c}
x_{\sigma(1)}\\
\vdots\\
x_{\sigma(n)}\\
\end{array}\right) + \left(\begin{array} {c}
\lambda_1\\
\vdots\\
\lambda_n\\
\end{array}\right).
$$

We may identify $\pft^{(n-1)}$ with $\mathbb{R}^{(n-1)}$ via the map
$$(x_1, \ldots, x_n) \mapsto (x_1 -x_n,\ldots, x_{n-1}-x_n)$$
and hence
$$\widehat{\phi_A}\left(\begin{array} {c}
x_{1}-x_n\\
\vdots\\
x_{n-1}-x_n\\
\end{array}\right) = \left(\begin{array} {c}
x_{\sigma(1)}-x_{\sigma(n)}\\
\vdots\\
x_{\sigma(n-1)}-x_{\sigma(n)}\\
\end{array}\right) + \left(\begin{array} {c}
\lambda_1-\lambda_n\\
\vdots\\
\lambda_{n-1}-\lambda_n\\
\end{array}\right).
$$
If $\sigma(n)=n$ then it is immediate that $\widehat{\phi_A}$ is an affine linear map on $\mathcal{P}C(E)$ regarded as a subset of $\mathbb{R}^{n-1}$. Suppose that $\sigma(n)\neq n$. Then $\sigma(k)=n$ for some $k \in\{1, \ldots,n-1\}$. Let $B$ be the $(n-1)\times (n-1)$ matrix given by

$$B_{i,j} = \begin{cases}
1 & \mbox{ if } i \neq k \mbox{ and } j=\sigma(i),\\
-1& \mbox{ if } j=\sigma(n),\\
0&\mbox{ otherwise.}\\
\end{cases}$$
Then it is easy to check that
$$\widehat{\phi_A}\left(\begin{array} {c}
x_{1}-x_n\\
\vdots\\
x_{n-1}-x_n\\
\end{array}\right) = B \cdot \left(\begin{array} {c}
x_{1}-x_n\\
\vdots\\
x_{n-1}-x_n\\
\end{array}\right) + \left(\begin{array} {c}
\lambda_1-\lambda_n\\
\vdots\\
\lambda_{n-1}-\lambda_n\\
\end{array}\right).$$
\end{proof}

We shall use Theorem~\ref{thm_upstairs} to prove our main result; that is, 
that every maximal subgroup $H_E$ is isomorphic to a direct product 
$\mathbb{R} \times \Sigma$, where $\Sigma$ is a subgroup of $S_n$. We 
shall require the following two lemmas, concerning units that commute with 
elements of $M_n(\ft)$.

\begin{lemma}\label{lem_uniteig} Let $A \in M_n(\ft)$ and let $G \in M_n(\trop)$
be a unit which commutes with $A$. Then $G$ has only one eigenvalue.
\end{lemma}

\begin{proof} Suppose false. It is easy to see that some power $G^k$ of $G$ is a diagonal matrix. Because $G$ has distinct eigenvalues, so does $G^k$, so $G^k$
cannot be a scaling matrix. But now it is easy to see that $G^k$ cannot commute with any matrix in $M_n(\ft)$, which contradicts the assumption that $G$ commutes with $A$.
\end{proof}

\begin{lemma}\label{lem_commute}Let $A \in M_n(\ft)$ and let $G$ and $H$ be a units which commute with
$A$. Suppose also that $G$ and $H$ have maximal cycle mean $0$. Then $GH$ has maximal cycle mean $0$.
\end{lemma}

\begin{proof}
It follows from Lemma~\ref{lem_uniteig} that every cycle in the graphs 
corresponding to $G$ and $H$ has mean weight $0$. In particular, the
(classical) sum 
of all the finite entries in $G$ is equal to $0$, and the sum of the 
finite entries in $H$ is equal to $0$. By a simple calculation, the 
sum of the finite entries in the product $GH$ must also be $0$. Now, 
since the product $GH$ also commutes with $A$, applying 
Lemma~\ref{lem_uniteig} again yields that every cycle in $GH$ has the same 
average weight. Since the sum of the entries (which is 0) is a weighted 
sum of the cycle means, and the cycle means are all the same, we deduce 
that the cycle means are all 0. In particular, the maximum cycle mean is 
0.
\end{proof}

\begin{theorem}
Let $E$ be an idempotent of rank $n$ in $M_n(\ft)$ and let $G_E$ denote the
group of units commuting with $E$. Define
$R = \{\lambda \otimes I_n: \lambda \in \ft\}$
and $\Sigma = \{ G \in G_E: G \mbox{ has eigenvalue } 0\}$. Then $R$
and $\Sigma$ are subgroups of $G_E$, $R \cong \mathbb{R}$, $\Sigma$ is
a finite group embeddable in the symmetric from $S_n$, and $G_E \cong R \times \Sigma$.
\end{theorem}

\begin{proof}
It is easy to see that $R \unlhd G_E$. Moreover, the only diagonal matrices which commute with $E$ are those contained in $R$. Since $I_n$ has eigenvalue $0$ we have $I_n \in \Sigma$. Let $A, B \in \Sigma$. By
Lemma~\ref{lem_commute} we see that $AB \in \Sigma$ and it is also clear
that $A^{-1} \in \Sigma$, giving that $\Sigma \leq G_E$. Thus
$R, \Sigma \leq H_E$ and it is clear that $R \cap \Sigma = \{I_n\}$.
Let $G \in G_E$. By Lemma~\ref{lem_uniteig}, $G$ has a unique eigenvalue,
$\lambda$ say. Hence we may write $G=\lambda \otimes I_n \otimes G_{\lambda}$, where $G_{\lambda} \in \Sigma$. In other words, $G_E = R\Sigma$. It is clear that every element of $R$ commutes with every element of $\Sigma$, giving $G_E = R \times \Sigma$.

It is clear that $R \cong \mathbb{R}$. Moreover, from the definition of
matrix multiplication it is easy to show that the map
$$\phi \ : \ G_E \rightarrow S_n, \ D(\lambda_1, \ldots, \lambda_n) P_{\sigma} \mapsto \sigma$$
is a homomorphism of groups, with kernel the set of all diagonal matrices that commute with $E$. Thus ${\rm ker} \phi = R$ and hence $\Sigma \cong G_E/R \cong {\rm Im} \phi \leq S_n$.
\end{proof}

\begin{corollary}
Let $E$ be an idempotent of rank $n$ in $M_n(\ft)$ and let $H_E$ denote the $\GreenH$-class of $E$. There exists an element $x \in C(E)$ such that $x$ is an eigenvector for all $A \in H_E$.
\end{corollary}

\begin{proof}
Let $S = \{\lambda \otimes E: \lambda \in \ft\}$ and $R = \{\lambda \otimes I_n: \lambda \in \ft\}$. Since
the map $\gamma: G_E \rightarrow H_E$ given by $\gamma(G) = EG$ is an isomorphism of groups, we see that $S=\gamma(R)$. Hence $S$ is a normal subgroup of $H_E$ which is isomorphic to $\mathbb{R}$ and $H_E / S$ is isomorphic to  ${\rm Im} \phi \leq S_n$. It is clear that $S$ acts trivially on $\mathcal{P}C(E)$. Thus the quotient group $H_E / S$ acts on $\mathcal{P}C(E)$ by affine linear transformations.  Since $H_E / S$ is a finite group, we may apply a theorem of Day \cite[Theorem~1]{Day61} to find a fixed point $x \in \mathcal{P}C(E)$ common to all elements of $H_E / S$. In other words, $x$ is an eigenvector for all elements of $H_E$.
\end{proof}

\begin{corollary}
\label{cor_directprod}
Let $H$ be a maximal subgroup of $M_n(\ft)$. Then $H$ is isomorphic to a direct product of the form $\mathbb{R} \times \Sigma$ for some $\Sigma \leq S_n$.
\end{corollary}

\bibliographystyle{plain}

\def\cprime{$'$} \def\cprime{$'$}

\end{document}